\documentclass[11pt, a4paper,leqno]{amsart}
\usepackage{amsmath,amsthm,amscd,amssymb,graphics}
\usepackage{latexsym}
\usepackage[colorlinks,citecolor=red,pagebackref,hypertexnames=false]{hyperref}

\usepackage{tikz}
\usetikzlibrary{decorations.pathmorphing,shapes}
\usepackage{enumerate}

\usepackage{verbatim}

\newtheorem{theorem}{Theorem}
\newtheorem{corollary}[theorem]{Corollary}
\newtheorem{conjecture}[theorem]{Conjecture}
\newtheorem{definition}[theorem]{Definition}
\newtheorem{proposition}[theorem]{Proposition}
\newtheorem{lemma}[theorem]{Lemma}

\begin{document}

\title{The two-dimensional small ball inequality and binary nets}
\author{Dmitriy Bilyk}
\address{School of Mathematics, University of Minnesota, Minneapolis, MN, 55455}
\email{dbilyk@math.umn.edu}
\author{Naomi Feldheim}
\address{Institute for Mathematics and its Applications, University of Minnesota, Minneapolis, MN, 55455}
\email{trinomi@gmail.com}

\maketitle

\begin{abstract}{In the current paper we present a new proof of the small ball inequality in two dimensions. More importantly, this new  argument, based on an approach inspired by lacunary Fourier series,  reveals the first formal connection between this inequality and discrepancy theory, namely the construction of two-dimensional binary nets, i.e. finite sets which are perfectly distributed with respect to dyadic rectangles. This relation allows one to generate all possible point distributions  of this type.  In addition, we outline a potential approach to the higher-dimensional small ball inequality by a dimension reduction argument. In particular this gives yet another proof of the two-dimensional signed (i.e. coefficients $\pm 1$) small ball inequality by reducing it to a simple one-dimensional estimate. However, we  show that  an analogous estimate fails to  hold for arbitrary coefficients.}
\end{abstract}

\section{Introduction}

The {\it{small ball inequality}} in harmonic analysis and {\it{discrepancy function estimates}} in discrepancy theory are known to be closely related to each other, see e.g. \cite{bl1}. However, so far this interplay was known only on a  heuristic level -- it manifested itself through common methods and arguments, but  no formal implications between the two facts have been available. In the present paper, we provide the first known formal connection between the small ball inequality and discrepancy theory. We give  a new proof of the two-dimensional version of the  small ball inequality in \S\ref{s.2d}, which in turn leads to  demonstrate  that  extremal sets of this inequality are precisely the binary nets -- a well-known class of low-discrepancy sets, see  \S\ref{s.nets}.

\subsection{Small ball inequality}
The {\it{small ball inequality}} is a lower bound for the supremum norm of ``hyperbolic'' sums of multiparameter Haar functions. Besides being interesting in its own right, it has important connections (both formal and heuristic) to approximation theory (metric entropy of mixed smoothness function classes), probability theory (small ball probabilities for the multiparameter Gaussian processes), and discrepancy theory (lower bounds for the sup-norm of the discrepancy function), see \cite{bilyk1,bilyk,bl1,talagrand,teml}.

To set the stage, let $\mathcal D$ denote the set of dyadic intervals in $[0,1)$, i.e. intervals of the form $\big[ m2^{-k}, (m+1)2^{-k}  \big)$, where $k\in \mathbb N_0$ and $m=0,1,...,2^k -1$. For a dyadic interval $I\in\mathcal D$, its left and right halves are also dyadic, and the Haar function associated to $I$ is defined as 
\begin{equation}
h_{I} (x) =  - {\mathbf 1}_{I_{left}} (x) + {\mathbf 1}_{I_{right}} (x).
\end{equation}
Notice that we normalize it in   $L^\infty$, not $L^2$.

In higher dimensions, we consider the family $\mathcal D^d$ of dyadic rectangles (boxes) in $[0,1)^d$  of the form $R= R_1 \times \dots \times R_d$ with $R_i \in \mathcal D$. The multiparameter Haar function associated to $R\in \mathcal D^d$ is then defined as a coordinatewise product of one-dimensional Haar functions:
\begin{equation}
h_R (x) = h_{R_1} (x_1) \cdot  \dots \cdot h_{R_d} (x_d ).
\end{equation}

The small ball conjecture deals with ``hyperbolic" linear combinations of Haar functions, i.e. sums of $h_R$ supported by rectangles of fixed volume $|R|=2^{-n}$. We shall use the symbol ``$\lesssim$" to mean ``less than a constant multiple of", i.e. $A \lesssim B$ means that there exists an absolute constant $C>0$ such that $A \le CB$. The implicit constant $C$  may depend on the dimension, but is independent of the scale $n$ and the choice of coefficients.

\begin{conjecture}\label{sbc} Let $d\ge 2$.  

 \noindent $\bullet$ {\bf{The small ball conjecture:}}  For each scale $n\in \mathbb N_0$ and any  coefficients $\alpha_R \in \mathbb R$ 
\begin{equation}\label{e.sbc}
n^{\frac{d-2}{2}} \cdot \bigg\| \sum_{R\in \mathcal D^d:\, |R| = 2^{-n} }  \alpha_R  h_R \bigg\|_\infty \gtrsim 2^{-n} \sum_{|R|=2^{-n} } |\alpha_R |.
\end{equation}

\noindent $\bullet$ {\bf{The signed small ball conjecture:}}  For each scale $n\in \mathbb N_0$ and any  coefficients $\varepsilon_R \in \{-1,+1\}$ 
\begin{equation}\label{e.ssbc}
 \bigg\| \sum_{R\in \mathcal D^d:\, |R| = 2^{-n} }  \varepsilon_R  h_R \bigg\|_\infty \gtrsim n^{\frac{d}2}.
\end{equation}

\end{conjecture}

The point of interest in the conjecture is the precise exponent of $n$. Inequalities \eqref{e.sbc} and \eqref{e.ssbc} with $n^{\frac{d-1}{2}}$  (in place of $n^\frac{d-2}{2}$ or, respectively, $n^\frac{d}{2}$) hold even for the $L^2$ norm in place of $L^\infty$  and are, in fact, very simple consequences of the Cauchy--Schwartz inequality. The exponent $d-1$ is very natural in this setting as it is the number of free parameters imposed by the condition $|R|=2^{-n}$. The conjecture states that for the supremum norm one can gain a factor of $\sqrt{n}$ over the $L^2$ estimate. The sharpness of the conjecture may be demonstrated by choosing random coefficients.

The conjecture holds in dimension $d=2$, and we shall discuss this case in more detail in the next section. In higher dimensions only partial results are known \cite{bl,blv,blv1,blpv,karslidis}. It is easy to see that \eqref{e.ssbc} is a particular case of \eqref{e.sbc}. Although the  {\it{signed}} case previously did not seem to have direct applications to other problems, it is an important ``toy model"  of the problem, which presents significant structural simplifications, while preserving most important obstacles \cite{blv1,blpv}. In \S \ref{s.nets} we shall establish a formal link between  the two-dimensional signed small ball inequality and discrepancy theory: we demonstrate that the extremal sets generated by this estimate yield an important class of low-discrepancy distributions.

\subsection{Discrepancy theory and nets} Discrepancy function quantifies equidistribution of a finite set in the unit  cube. Let $\mathcal P_N $  be a set of $N$ points in $[0,1)^d$.  Its discrepancy function is defined as 
\begin{equation}
D_N ( x_1,...,x_d ) = \# \mathcal P_N \cap [0,x_1)\times...\times [0,x_d) - N x_1\cdot  ... \cdot  x_d,
\end{equation}
i.e. the difference between the true and expected numbers of points in the box $ [0,x_1)\times...\times [0,x_d)$.
The basic principle of {\it{irregularities of distribution}} states that various norms of this  function necessarily have to grow with $N$. The known bounds support this principle:

  \noindent $\bullet$ Roth \cite{R54}, Schmidt \cite{S77}: In all dimensions $d\ge 2$, for all $p\in(1,\infty)$, for any $\mathcal P_N \subset [0,1)^d$
\begin{equation}
\| D_N \|_p \gtrsim (\log N)^{\frac{d-1}{2}}.
\end{equation}
\noindent $\bullet$ Schmidt \cite{S72}: In  dimension $d= 2$, for any $\mathcal P_N \subset [0,1)^2$
\begin{equation}
\| D_N \|_\infty \gtrsim \log N.
\end{equation}
\noindent $\bullet$ Bilyk, Lacey, Vagharshakyan \cite{bl,blv}: In all  dimensions $d\ge 3$, there exists $\eta_d \in (0, 1/2) $ such that  for any $\mathcal P_N \subset [0,1)^d$
\begin{equation}
\| D_N \|_\infty \gtrsim (\log N)^{\frac{d-1}{2} + \eta_d}.
\end{equation}

The best known constructions of {\it{low-discrepancy sets}}, see e.g. \cite{DP}, satisfy 
\begin{equation}\label{e.ld}
\| D_N \|_\infty \lesssim (\log N)^{d-1}.
\end{equation}
Hence one can see that the sharp rate of growth is known only in dimension $d=2$. There is no consensus as to the conjecture about the right asymptotics of $\|D_N\|_\infty$ for $d\ge 3$ (sometimes referred to as {\it{star-discrepancy}}).
\begin{conjecture}
In dimensions $d\ge3$, for any $N$-point set $\mathcal P_N \subset [0,1)^d$ the sharp lower bound for the star-discrepancy is 
\begin{eqnarray}
\label{e.dc1} \|D_N \|_\infty & \gtrsim & (\log N)^{d-1} \quad \textup{ or}\\
\label{e.dc2} \| D_N \|_\infty & \gtrsim  & (\log N)^{d/2}.
\end{eqnarray}
\end{conjecture}  
While the first conjecture \eqref{e.dc1} is supported by the bounds for the best known low-discrepancy sets \eqref{e.ld}, the second one \eqref{e.dc2} stems from the {\it{small ball inequality}} \eqref{e.sbc}: in the {\it{signed}} case \eqref{e.ssbc} the similarity becomes especially apparent. While no formal implications have been proved, hyperbolic sums of Haar functions serve as a model for $D_N$, and \eqref{e.sbc} can be viewed as the linear term in the discrepancy estimates, see e.g. \cite{bilyk1,bilyk,bl1} for an extensive discussion.

An important class of examples  of low-discrepancy sets is given  by the so-called nets. 
\begin{definition} Let  $t\ge 0$, $m\ge 0$, $d\ge 1$ be integers. A finite set $\mathcal P\subset [0,1)^d$ of $N=2^m$ points is called a dyadic or binary $(t,m,d)$-net if every dyadic box of volume $2^{t-m}$ contains precisely $2^t$ points of $\mathcal P$. 
Similarly, nets can be defined in other integer bases $b\ge 2$ besides binary.
\end{definition}

The parameter $t$ is called  {\it{deficiency}} or the quality parameter of the net. The case $t=0$ is the most interesting: in this situation every $b$-adic box of size $\frac1{N}$ contains exactly one point of $\mathcal P$. However, it is well known that perfect nets with $t=0$ exist only when $b\ge d-1$. In particular, perfect dyadic nets exist only in dimension two and three. Nevertheless, in every dimension $d\ge 2$ and in every base $b \ge 2$, there exist nets with deficiency depending  only on the dimension. Such nets by definition have  discrepancy zero with respect to $b$-adic boxes of size comparable to $\frac1{N}$. This fact can be extrapolated to arbitrary rectangles to show that these sets satisfy \eqref{e.ld} $\| D_N \|_\infty \lesssim (\log N)^{d-1} \approx m^{d-1}$, where the implied constant depends only on $d$, $b$, and $t$. For encyclopedic information about nets the reader is referred to the recent book \cite{DP}.

In this paper we look more closely at the $(0,m,2)$-nets in base $b=2$, i.e. two-dimensional perfect dyadic nets with $N=2^m$ points. In \S \ref{s.nets} we shall show that all such nets arise precisely as extremal sets of the two-dimensional signed small ball inequality. This characterizes all two-dimensional $(0,m,2)$-nets and gives an easy way to count the total number of such nets. In addition, this provides another important (and this time formal) connection  between the small ball conjecture and discrepancy estimates. Our construction also easily generalizes to $b$-adic nets.

We shall restrict our attention to nets, in which all points have coordinates of the form $k b^{-m}$, where $k=0,1,..,b^{-m}-1$. In fact, every $(0,m,2)$-net can be transformed into a net of this type using the map $x \rightarrow b^{-m} \lfloor{b^m x} \rfloor$. We shall not distinguish the nets which are mapped into the same set by this transformation and shall treat them as equivalent.

We finish the introduction by describing some standard constructions of nets in dimension $d=2$. Perhaps the best known example is  the famous ``digit-reversing" Van der Corput set with $N=2^{m }$ points. This set consists of all $2^{m}$ points of the following form
 \begin{equation}\label{e.vdc}
 \big( 0.x_1 x_2 \dots x_{m-1} x_{m}, \, 0.x_{m} x_{m-1} ... x_2 x_1 \big),
 \end{equation}
 where the coordinates are written in the binary form, i.e. the digits $x_i=0$ or $1$. It is easy to see that each dyadic rectangle $R$ of area $|R|=2^{-m}$ contains precisely one point of this set, i.e. it is a $(0,m,2)$-net in base $b=2$, and $b$-adic modifications are obvious.
 
 A simple standard modification of the dyadic  Van der Corput set that preserves the net property is the digit-shift defined as follows. Fix $\sigma_k \in \{0,1\}$, $k=1,2,\dots,m$. The digit-shifted Van der Corput set consists of $2^{m}$ points of the form 
 \begin{equation}\label{e.dsvdc}
 \big( 0.x_1 x_2 \dots x_{m-1} x_{m}, \, 0.y_{m} y_{m-1} ... y_2 y_1 \big),
 \end{equation}
 where $y_k = x_k + \sigma_k (\textup{mod } 2)$. In other words, after reversing the order of the digits, we also change  the digits in a fixed set of positions (those where $\sigma_k=1$). More general modifications (digit-scrambling) are also used.

\section{The two-dimensional small ball inequality}\label{s.2d}


The small ball conjecture \eqref{e.sbc} is known to be true  in dimension $d=2$. There are at least  two previous proofs of this fact: one due to M.~Talagrand \cite{talagrand}, and one by V.~Temlyakov \cite{teml}, both dating from the early 1990s. The latter proof uses Riesz products -- a technique which was originally used in the proof of Sidon's theorem on lacunary Fourier series, which has an obvious similarity to the small ball inequality, see \S \ref{s.lac}. In this section we present a new proof of this  inequality.

\begin{theorem}[The small ball inequality]\label{t.sbi2}
In dimension $d=2$, for any scale $n\in N_0$ and any coefficients $\alpha_R \in \mathbb R$, the following inequality holds
\begin{equation}\label{e.sbi2}
\bigg\|  \sum_{|R| = 2^{-n} } \alpha_R h_R \,\, \bigg\|_\infty \ge 2^{-n}  \sum_{|R| = 2^{-n} } \big| \alpha_R \big|.
\end{equation}
\end{theorem}

We shall actually  first give a proof of its weaker variant -- the {\emph{signed small ball inequality}}, see \eqref{e.ssbc}, i.e. 
a version of \eqref{e.sbi2} with all coefficients equal to $\pm 1$. Observe that if we fix the area of  a two-dimensional dyadic rectangle $|R|=2^{-n}$, there are exactly  $n+1$ different possible  shapes that such a rectangle may have. Indeed, the first side may have length $|R_1| = 2^0$, $2^{-1} $, ..., $2^{-n}$, and the length of the second side is determined automatically.  Therefore in this case $ \displaystyle{  \sum_{|R| = 2^{-n} } \big| \alpha_R \big|  = 
2^n (n+1) }$, and Theorem \ref{t.sbi2} becomes:

\begin{theorem}[The signed small ball inequality]\label{t.ssbi2}
In dimension $d=2$, for any scale $n\in N_0$ and any coefficients $\varepsilon_R = \pm 1$, the following inequality holds
\begin{equation}\label{e.ssbi2}
\bigg\|  \sum_{|R| = 2^{-n} } \varepsilon_R h_R \,\,  \bigg\|_\infty = n+1.
\end{equation}
\end{theorem}

Obviously, at each point $x\in [0,1]^d$ the sum on the left-hand side contains exactly $n+1$ terms equal to $\pm 1$, so the upper bound  trivially holds, and the lower bound is the core of the matter (hence we keep the name ``inequality" in  \eqref{e.ssbi2}). 
Simple combinatorial structure of the signed case makes the argument very elegant and the idea of our proof particularly transparent. 

\subsection{Proof of the signed small ball inequality (Theorem \ref{t.ssbi2}).} We shall assume that $n$ is odd, the other case being completely analogous. For $k=0$, $1$,..., $n$ define
\begin{equation}
\mathcal D^2_k = \{ R = R_1 \times R_2 \in \mathcal D^2:\, |R_1|=2^{-k},\, |R_2| = 2^{-(n-k)} \},
\end{equation}
i.e. this is the set of  $2^{-k}$-by-$2^{-(n-k)}$ dyadic boxes.\\

For each $k=\frac{n+1}2$,..., $n-1$, $n$, let us set
\begin{equation}
F_k (x) = \sum_{R\in \mathcal D^2_k} \varepsilon_R h_R(x)  + \sum_{R\in \mathcal D^2_{n-k}} \varepsilon_R h_R(x),
\end{equation}
i.e. it is the part of our sum which contains rectangles of dimensions $2^{-k} \times 2^{-(n-k)}$ and the symmetric $2^{-(n-k)} \times 2^{-k}$. In particular, $F_k$ is constant on dyadic squares of side length $2^{-(k+1)}$.

Moreover, one can easily see that on each dyadic square of side length $2^{-k}$, up to rotations and reflections, the function $F_k$ has the form shown in Fig. \ref{f.2002}, i.e. it takes values $2$ and $-2$ in two opposite quarters and zero in the other two.

\begin{figure}
\begin{tikzpicture}
\draw (0,0) rectangle (4,2);
\draw (2,0)--(2,2);
\draw (0,1)--(4,1);
\draw[red,dashed, ultra thick] (0,0) rectangle (2,2);
\node at (1,0.5) {-1};
\node at (1,1.5) {1};
\node at (3,0.5) {1};
\node at (3,1.5) {-1};
\draw (4.8,1)-- (5.2,1);
\draw (5,0.8) -- (5,1.2);
\end{tikzpicture}
\hspace{0.4cm}
\begin{tikzpicture}
\draw (0,0) rectangle (2,4);
\draw (0,2)--(2,2);
\draw (1,0)--(1,4);
\draw[red,dashed, ultra thick] (0,0) rectangle (2,2);
\node at (0.5,1) {-1};
\node at (1.5,1) {1};
\node at (0.5,3) {1};
\node at (1.5,3) {-1};
\draw [->,
line join=round, decorate, decoration={
    zigzag,
    segment length=4,
    amplitude=.9,post=lineto,
    post length=4pt
}]  (2.5,1) -- (3,1);
\end{tikzpicture}
\hspace{0.3cm}
\begin{tikzpicture}
\filldraw[gray!30!white]  (1,1)rectangle (2,2);
\draw (0,0) grid (2,2);
\draw[red,dashed, ultra thick] (0,0) rectangle (2,2);
\node at (0.5,0.5) {-2};
\node at (1.5,0.5) {0};
\node at (.5,1.5) {0};
\node at (1.5,1.5) {2};
\end{tikzpicture}
\caption{
The sum of two Haar functions of symmetric rectangles determines the values of 
$F_k$ on a dyadic cube of scale  $2^{-k}$.
Since each Haar function is multiplied by $1$ or $-1$ in the sum composing $F_k$, this square may be rotated and reflected.
The chosen sub-square, on which $F_k$ takes the value $2$, is shaded. \label{f.2002} }
\end{figure}
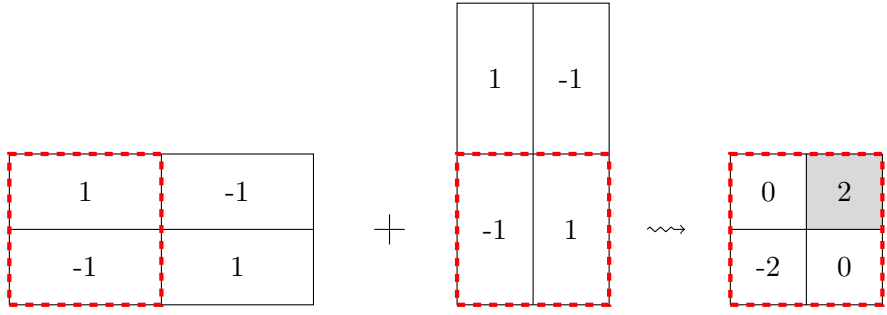


We start with $k=\frac{n+1}{2}$ and keep the cubes of scale $2^{-(k+1)}$ on which $F_k = 2$, discarding all the others. Notice that there are exactly $2^{n+1}$ such cubes.

We then proceed inductively increasing $k$ by one at each step.  On every previously chosen cube of scale $2^{-(k+1)}$ the function $F_{k+1}$ takes values $0$, $-2$, and $2$, as in Figure \ref{f.2002}. We choose the subcubes on which $F_{k+1} = 2$ and move on to the next step.

In the end we have found $2^{n+1}$ cubes of area $2^{-2(n+1)}$ on which $F_k = 2$ for all $k=\frac{n+1}2$,..., $n-1$, $n$. Obviously, on each such cube
\begin{equation}
\sum_{|R| = 2^{-n} } \varepsilon_R h_R = \sum_{k=\frac{n+1}{2}}^n F_k  = n+1,
\end{equation}
which implies \eqref{e.ssbi2}.    $\square$

\subsection{Proof of the  small ball inequality for general coefficients (Theorem \ref{t.sbi2})} The proof in the  case of arbitrary coefficients, inequality \eqref{e.sbi2}, requires minimal modifications and an additional simple observation.\\

For each $k=\frac{n+1}2$,..., $n-1$, $n$, we similarly define
\begin{equation}
F_k (x) = \sum_{R\in \mathcal D^2_k} \alpha_R h_R(x)  + \sum_{R\in \mathcal D^2_{n-k}} \alpha_R h_R(x).\end{equation}

On each dyadic cube $Q$ of side length $2^{-k}$ the structure of $F_k$ can be described as follows. Let $R' \in \mathcal D^2_k$ and $R'' \in \mathcal D^2_{n-k}$ be such that $R' \cap R'' = Q$. Then $F_k$ equals $\big|\alpha_{R'} \big| + \big|\alpha_{R''} \big|$ in one quarter of $Q$, $ - \big( \big|\alpha_{R'} \big| +\big|\alpha_{R''} \big|\big)$ in the opposite quarter, and the values on the remaining two sub cubes are immaterial. 

Just as in the signed case, starting from $k=\frac{n+1}2$ we proceed inductively, at each step keeping the subcubes on which $F_k$ takes values $\big|\alpha_{R'} \big| + \big|\alpha_{R''} \big|$.

We make the following additional observation. At the initial step $k=\frac{n+1}2$ each rectangle $R\in \mathcal D^2_k$ or $\mathcal D^2_{n-k}$ contains exactly {\emph{two}} chosen cubes: indeed, every such $R$ consists of exactly two cubes of side length $2^{-k}$, and each of them in turn contains one chosen subcube. Moreover, because of the structure of the Haar function and the fact that we were choosing cubes on which $\alpha_R h_R (x) \ge 0$, these two cubes lie in opposite quarters of $R$.

This guarantees that at each following step dyadic boxes also contain exactly two chosen boxes in opposite quarters, because any rectangle $R\in \mathcal D^2_{k+1}$ intersects exactly two rectangles from $\mathcal D_k^2$ and will thus contain one previously chosen cube in each. We shall further choose a  subcube in each of those, which will in turn have to lie in opposite quarters of $R$ (see Figure~\ref{f.chosen}).

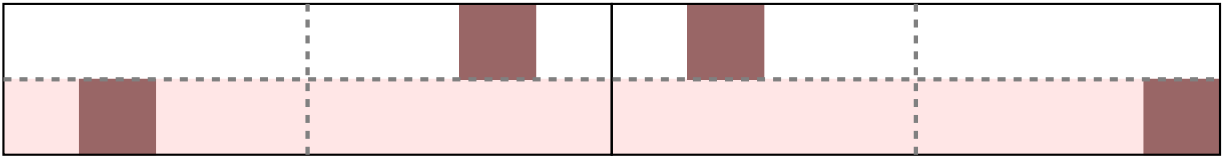
\begin{figure}
\begin{tikzpicture}
\filldraw[red!10!white] (0,0)rectangle (16,1);
\filldraw[gray!80!red] (1,0) rectangle (2,1);
\filldraw[gray!80!red] (15,0) rectangle (16,1);
\filldraw[gray!80!red] (9,1) rectangle (10,2);
\filldraw[gray!80!red] (6,1) rectangle (7,2);
\draw[thick] (0,0) rectangle (8,2);
\draw[thick] (8,0) rectangle (16,2);
\draw[gray,dashed,ultra thick] (4,0)--(4,2);
\draw[gray,dashed,ultra thick] (0,1)--(16,1);
\draw[gray,dashed,ultra thick] (12,0)--(12,2);

\end{tikzpicture}
\caption{An induction step in showing that each rectangle $R$ contains precisely two chosen squares, lying in two opposite quarters. \label{f.chosen} }
\end{figure}
In the end we will have chosen $2^{n+1}$ cubes $Q_1$, $Q_2$,..., $Q_{2^{n+1}}$ of area $2^{-2(n+1)}$ such that on each cube $Q_j$ we have
\begin{equation}
\sum_{|R|=2^{-n}} \alpha_R h_R (x) = \sum_{ R \supset Q_j} \big| \alpha_R \big|.
\end{equation}
In addition, every dyadic box $R$ with $|R|=2^{-n}$ contains precisely two cubes $Q_j$.

Estimating the maximum by the average we easily obtain 
\begin{align}
\bigg\|  \sum_{|R| = 2^{-n} } \alpha_R h_R \,\, \bigg\|_\infty & \ge \max_{j=1,\dots,2^{n+1}} \sum_{ R \supset Q_j} \big| \alpha_R \big|\\
\nonumber & \ge \frac{1}{2^{n+1}} \sum_{Q_j} \sum_{ R \supset Q_j} \big| \alpha_R \big|
 \,\,\, = \,\,\, \frac{1}{2^{n+1}} \sum_{R}  \big| \alpha_R \big|  \sum_{  Q_j \subset R} 1\\
 \nonumber & = 2^{-n} \sum_{R: |R| = 2^{-n}}  \big| \alpha_R \big|,
\end{align}
which proves \eqref{e.sbi2}. \,\,\, $\square$

\subsection{Similarities to lacunary Fourier series}\label{s.lac}
One  needs not look very scrupulously in order to notice the similarity of the small ball inequality \eqref{e.sbi2} to Sidon's theorem on lacunary Fourier series. 

\begin{theorem}[Sidon, 1927 \cite{sidon}]\label{t.sidon}
Let the sequence $\{n_k\}\subset \mathbb N$ be lacunary, i.e. there exists $\lambda >1$ such that for each $k\in \mathbb N$ we have $\displaystyle{\frac{n_{k+1}}{n_k} \ge \lambda}$. Then there exists a constant $C>0$ such that for any real coefficients $\alpha_k$, the following inequality holds
\begin{equation}\label{e.sidon}
\bigg\| \sum_k \alpha_k \sin 2\pi n_k x \bigg\|_\infty \ge C \sum_k \big| \alpha_k \big|.
\end{equation}
\end{theorem}

The similarity becomes somewhat natural if one realizes that the two-dimensional  Haar functions   on dyadic boxes $\big\{ R\in \mathcal D^2:\, |R|=2^{-n} \}$  essentially form a one-parameter family (one of the two parameters is eliminated by the condition $ |R|=2^{-n}$) and their frequencies are dyadic, i.e. lacunary with $\lambda =2 $.

As pointed out earlier, the proof of \eqref{e.sbi2} given by Temlyakov in \cite{teml} utilizes Riesz products and is very similar to the proof of Sidon's theorem. More precisely,  Riesz products by themselves prove \eqref{e.sidon} for the case $\lambda \ge 3$. (Both proofs are described in detail in \cite{bilyk1,bilyk}.) In fact, the cubes chosen in our argument are precisely the support of the Riesz product constructed in \cite{teml}. 

We would like to point out that our argument is  similar to yet another proof of Sidon's theorem with large lacunarity constant. Indeed, if we assume that, say $\lambda \ge 9$, we can run the following induction argument. At each step $k$ look at the set $\big\{ x\in [0,1):\, \alpha_k \sin 2\pi n_k x \ge \frac12 \big| \alpha_k \big| \big\}$. This set consists of periodically repeating intervals of length $\frac{1}{3 n_k}$. Each of these intervals contains at least $3$ full periods of $\sin 2\pi n_{k+1} x$, in particular at the next step we can find a complete interval of length $\frac1{3n_{k+1}}$ on which $ \alpha_{k+1} \sin 2\pi n_{k+1} x \ge \frac12 \big| \alpha_{k+1} \big|$ within the previously chosen interval.  One can immediately recognize this idea in our proof of  inequality \eqref{e.ssbi2}.

\section{Two-dimensional small ball inequality and  low-discrepancy sets}\label{s.nets}

We now take a closer look at the structure of the set of $2^{n+1}$ squares of side length $2^{-(n+1)}$  which was constructed in the proof of Theorem \ref{t.sbi2}  and \ref{t.ssbi2}, in other words the set where $\displaystyle{\sum_{|R| = 2^{-n} } \varepsilon_R h_R (x)}$ achieves the maximal value, or equivalently, the set on which all the terms of $\displaystyle{\sum_{|R| = 2^{-n} } \alpha_R h_R (x)}$ are nonnegative.

 It has already been observed in \cite{bilyk} that in the case of positive coefficients, i.e. when   $\alpha_R \ge 0$ (or $\varepsilon_R = + 1$)  for all $R\in \mathcal D^2$ with $|R|=2^{-n}$, this set coincides with the  ``digit-reversing" Van der Corput set  \eqref{e.vdc} with $N=2^{n+1}$ points (here we frivolously identify small squares with points, namely their lower left corners). Indeed, given a dyadic rectangle $R$ of dimensions $2^{-k}\times 2^{-(n-k)}$, it is easy to see that $h_R(x) = + 1$ precisely for those points $x = \big(x^{(1)}, x^{(2)}\big) \in R$, for which the $(k+1)^{st}$ binary digit of $x^{(1)}$ coincides with the $(n-k+1)^{st}$ digit of $x^{(2)}$, i.e. both are $0$ or both are $1$. \\

We now turn to the case when the signs of the coefficients are completely arbitrary. Looking back at the proof of Theorem \ref{t.sbi2}, we recall that we found that every rectangle $R\in \mathcal D^2$ of area $2^{-n}$ contains precisely $2$ of the squares from our resulting set. This means exactly that this set is a $(1,n+1,2)$-net. However, in addition we know that these squares live in opposite quarters of $R$. Thus each dyadic rectangle $R$ of area $2^{-(n+1)}$ contains precisely one chosen square. Indeed, take any dyadic  ``parent" of $R$ (vertical or horizontal): it contains two squares, but in opposite quarters, so exactly one of them lives in $R$. We therefore arrive at the following result.

\begin{theorem}\label{t.net}
Let   $\varepsilon_R = \pm 1$  for $R\in \mathcal D^2$ with $|R|=2^{-n}$. Then the set of points,  on which the sum $\displaystyle{\sum_{|R| = 2^{-n} } \varepsilon_R h_R (x)}$ achieves its maximal value of $n+1$, consists of $2^{n+1}$ dyadic squares of side length $2^{-(n+1)}$, whose lower left corners form a binary $(0,n+1,2)$-net in base $2$, i.e. every dyadic box of area $2^{-(n+1)}$ contains precisely one cube from  this set.
\end{theorem}

It is easy to see that this construction covers the standard examples of two-dimensional nets:

\noindent $\bullet$ As discussed above, when $\varepsilon_R = +1$ for all  $R\in \mathcal D^2$ with $|R|=2^{-n}$, the arising net is the Van der Corput set with $N=2^{n+1}$ points \eqref{e.vdc}.  

\noindent $\bullet$ If $\varepsilon_R$ depends only on the geometry of $R$, i.e. $\varepsilon_R = \varepsilon (|R_1|, |R_2|)$, in other words the signs are the same for the whole ``layer" of disjoint dyadic rectangles of the same shape, then one obtains a digit-shifted Van der Corput set \eqref{e.dsvdc}. It is easy to see that  the shift is given by a sequence $\sigma$, for which $\sigma_k = 1$  precisely when the boxes  $R\in \mathcal D^2_{k-1}$, i.e. with $|R_1|=2^{-(k-1)}$, have coefficient $\varepsilon_R = -1$.

\noindent $\bullet$ If the coefficients have product structure, i.e. for $R= R_1 \times R_2$ we have $\varepsilon_{R_1} \cdot \varepsilon_{R_2}$, then the construction yields the so-called Owen's scrambling (\cite{owen}; see \cite[page 396]{DP}) of the Van der Corput set. Incidentally, the signed small ball inequality for coefficients with product structure has been recently proved in all dimensions \cite{karslidis}.

When the signs of the coefficients are arbitrary and not structured, our construction  provides a  much wider range of  of nets. Moreover, {\it{all}} nets can be constructed via this procedure. 

\begin{proposition}\label{p.allnets}
Let $\mathcal P$ be a $(0,n+1,2)$-net in base $2$ in which all points have coordinates of the form $k 2^{-(n+1)}$, where $k=0,1,...,2^{n+1} -1$. Then there exists a choice of coefficients $\varepsilon_R = \pm 1$, $R\in \mathcal D^2$, $|R|=2^{-n}$, such that $\displaystyle{\sum_{|R| = 2^{-n} } \varepsilon_R h_R (x)}$ achieves its maximal value of $n+1$ on the set $$\mathcal P + \big[0,2^{-(n+1)}\big)^2,$$ i.e. on a set, which consists of $2^{n+1}$ dyadic squares of side length $2^{-(n+1)}$, whose lower left corners are precisely the points of $\mathcal P$.
\end{proposition}

\begin{proof}   By the net property, any rectangle $R\in \mathcal D^2$ with $|R|=2^{-n}$ contains exactly two points of $\mathcal P$, as it is a  union of two rectangles of volume $2^{-(n+1)}$. Since $R$ can be divided into dyadic "children" in two ways (horizontally and vertically),  these points lie in the opposite quarters of $R$. We can now choose the sign $\varepsilon_R$ so that these points lie in the set where $\varepsilon_R h_R = +1$.

We thus construct coefficients for all dyadic boxes of area $2^{-n}$.   By construction, at every point $p$ of the net $\mathcal P$ together with the adjacent small square, i.e. on  the set $\mathcal P + \big[0,2^{-(n+1)}\big)^2$, we have $\displaystyle{\sum_{|R| = 2^{-n} } \varepsilon_R h_R (x)} = n+1$.  But Theorem \ref{t.ssbi2} implies that this is the maximal value of the sum, and Theorem \ref{t.net} dictates that the set where this maximum is achieved does not contain anything else.
\end{proof}

As a side perk, this characterization gives a simple way to compute the number of different nets. Set $m= n+1$. The set  $\{ R\in \mathcal D^2:\, |R|=2^{-n}\}$ contains $(n+1)2^n = m 2^{m-1}$ elements, and each $\varepsilon_R$ takes one of two values. We thus obtain the following:

\begin{corollary}\label{c.numbernets}
Let $m\in \mathbb N$. The number of different dyadic $(0,m,2)$-nets, in which all points have coordinates of the form $k2^{-m}$, $k=0,1,...,2^{m}-1$ is $\displaystyle{2^{m2^{m-1}}}$.
\end{corollary}

This fact has been previously established by Xiao \cite{xiao1,xiao2}, who also provided a different characterization of $(0,m,2)$-nets. In addition, almost simultaneously with us, Pillichshammer and coauthors  \cite{pil} obtained yet another algorithm of generating all dyadic $(0,m,2)$-nets.

\subsection{Extensions to general integer bases $b\ge 2$.} It is fairly straightforward to generalize these results to $b$-adic nets. Rather than using standard $b$-adic Haar functions, we shall utilize a specially constructed collection of functions, which generalizes dyadic Haar functions in a slightly different way adapted to our goal.

For an integer base $b\ge 2$, we define the system of $b$-adic subintervals of $[0,1)$ as $$\mathcal D_b = \big\{ \big[ k b^{-m}, (k+1)b^{-m} \big):\, m\in \mathbb N_0, \, k =0,1,\dots,b^{m}-1 \big\}.$$ Then $\mathcal D_b^d$ is the collection of $b$-adic boxes in $[0,1)^d$.

A given box $R \in \mathcal D_b^2$ of dimensions $b^{-m_1} \times b^{-m_2}$ can be represented as a union of a $b\times b$ array of $b$-adic boxes of dimensions $ b^{-(m_1+1)} \times b^{-(m_2+1)}$. We define the family of functions $\mathcal H_R$ as follows. The function $\phi_R\in \mathcal H_R$ if and only if the following conditions hold:

\noindent $\bullet$ $\phi_R$ takes values $\pm 1$ on $R$ and vanishes outside $R$.

\noindent $\bullet$ $\phi_R$ is constant on $b$-adic subboxes  of $R$ of dimensions $ b^{-(m_1+1)} \times b^{-(m_2+1)}$.

\noindent $\bullet$ In each row and in each column of the $b\times b$ array of $b$-adic subboxes of $R$ of dimensions $ b^{-(m_1+1)} \times b^{-(m_2+1)}$, there is  exactly one subbox, on which $\phi_R = +1$.

Obviously $\mathcal H_R$ contains $b!$ different functions. In addition, observe that when $b=2$, $\mathcal H_R$ consists of precisely 2 functions: $\pm h_R$. One can easily check that  our arguments directly extend to the $b$-adic  case. We summarize the results in the following theorem. For the sake of brevity, we shall call $(t,m,d)$-nets   in which all points have coordinates of the form $kb^{-m}$, $k=0,1,\dots, b^m-1$, {\it{standard}}.

\begin{theorem}\label{t.badic}
Fix the scale $m \in \mathbb N$ and an integer base $b\ge 2$. For each $b$-adic box $R\in \mathcal D_b^2$ with $|R| = b^{-(m-1)}$, choose a function $\phi_R \in \mathcal H_R$. Then
\begin{enumerate}[(i)]
\item A $b$-adic analogue of the signed small ball inequality holds:
\begin{equation}
\max_{x\in [0,1)^2} \sum_{\substack{R\in \mathcal D_b^2\\ |R| = b^{-(m-1)} }} \phi_R (x)  \,\,   = m.
\end{equation}
\item The set on which the maximum above is achieved has the form 
\begin{equation}
\mathcal P + \big[0, b^{-m} \big)^2,
\end{equation}
where $\mathcal P$ is a standard $(0,m,2)$-net in base $b$. 
\item Every standard $(0,m,2)$-net $\mathcal P$ in base $b$ may be obtained in this  way, i.e. for every such $\mathcal P$ there exists a choice of $\phi_R \in \mathcal H_R$ so that the set, where $\sum_{ |R| = b^{-(m-1)} } \phi_R$ attains its maximal value, is precisely $\mathcal P + \big[0, b^{-m} \big)^2$.
\item The number of different standard $(0,m,2)$-nets in base $b$ is $\displaystyle{(b!)^{mb^{m-1}}}$.
\end{enumerate}
\end{theorem}

We point out that part $(iv)$ has been previously proved by Xiao \cite{xiao1,xiao2}.

\section{Dimension reduction for the signed small ball inequality}

The signed small inequality in dimension $d$ may be deduced from a similar inequality in dimension $d-1$, where the summation condition $|R|=2^{-n}$ is replaced by the condition $|R|\ge 2^{-n}$.

\begin{lemma}\label{l.reduce}
Let $d\ge 2$. Assume that in dimension  $d' = d-1 \ge 1$ for all coefficients $\varepsilon_R= \pm 1$ we have the following inequality:
\begin{equation}\label{e.l2a}
\bigg\|  \sum_{|R| \ge 2^{-n}  } \varepsilon_R h_R \,\, \bigg\|_\infty \gtrsim n^{\frac{d'+1}{2}} = n^{\frac{d}{2}}.
\end{equation}

Then in dimension $d\ge 2$ for all coefficients $\varepsilon_R= \pm 1$ we have
\begin{equation}\label{e.l2}
\bigg\|  \sum_{|R| = 2^{-n} } \varepsilon_R h_R \,\, \bigg\|_\infty \gtrsim n^{\frac{d}{2}}.
\end{equation}

\end{lemma}

\begin{proof} The proof of this fact is based on the restriction of the sum in \eqref{e.l2} to a hyperplane. Indeed, write $x\in [0,1)^d$ as $x= (x_1, x')$ with $x_1 \in [0,1)$, $x' \in [0,1)^{d-1}$, and similarly write $R \in \mathcal D^d$ as $R= R_1 \times R'$, where $R_1\in \mathcal D$ and $R' \in \mathcal D^{d-1}$. We have  $|R|= 2^{-n}$ if and only if $|R'|\ge 2^{-n}$.   For a fixed $x_1 \in [0,1)$ and $k\ge 0$, there is exactly one dyadic interval $R_1 $ of length $2^{-k}$ containing $x_1$.

Therefore, with $x_1$ fixed, the sum in \eqref{e.l2} becomes $$ \sum_{|R| = 2^{-n} } \varepsilon_R h_R (x)  = \sum_{|R'| \ge 2^{-n} } \sum_{\substack{R_1\ni x_1 \\ |R_1|= 2^{-n}/ |R'| }} \Big[\varepsilon_{R_1 \times R'} h_{R_1} (x_1) \Big] h_{R'} (x')  = \sum_{|R'| \ge 2^{-n} }  \varepsilon'_{R'} h_{R'} (x').$$ Here we used the fact that the sum in the middle contains just one non-zero term  and we have set  $\varepsilon'_{R'} =  \varepsilon_{R_1 \times R'} h_{R_1} (x_1) = \pm 1$, where $R_1$ is uniquely determined by $R'$ and $x_1$. Since for fixed $x_1$, we have $\| f \|_{L^\infty (x) } \ge \| f (x_1, \cdot ) \|_{L^\infty (x')}$, the conclusion of the lemma follows.
\end{proof}

This lemma yields yet another very simple proof of the two-dimensional signed small ball inequality. Indeed, the one-dimensional estimate 
\begin{equation}\label{e.s1d}
\displaystyle{   \bigg\|  \sum_{|R| \ge 2^{-n}} \varepsilon_R h_R   \bigg\|_\infty \ge n+1 }
\end{equation}
 is almost obvious: it can be proved by choosing a nested sequence of dyadic intervals $R^{(k)}$ with $|R^{(k)}|= 2^{-k}$, $k=0,1,\dots,n$, on which all $\varepsilon_{R^{(k)}} h_{R^{(k)}} = +1$.\\
 
 Moreover, in the two-dimensional case, the converse implication in Lemma \ref{l.reduce} also holds, which follows from our proof of \eqref{e.ssbi2}.  Consider the one dimensional sum $\sum_{|R_2|\ge 2^{-n}} \varepsilon_{R_2} h_{R_2} (x_2)$.  For any two-dimensional  box $R = R_1 \times R_2$ with $|R| =2^{n}$ and such that $R_1$ contains zero (in which case $ h_{R_1} (0) = +1$), set $\varepsilon_R = \varepsilon_{R_2} $, and define the  signs arbitrarily for other boxes. Since the extremal set of the two dimensional sum  is a net, it contains a $2^{-(n+1)} \times 2^{-(n+1)}$ square in the box $\big[0,2^{-(n+1)} \big) \times [0,1)$, i.e. we know that the maximum is achieved, in particular when $x_1 = 0$. We then have 
 \begin{equation}
  \bigg\|  \sum_{|R_2| \ge 2^{-n}} \varepsilon_{R_2} h_{R_2} (x_2)  \bigg\|_\infty = \bigg\|  \sum_{\substack{R=R_1 \times R_2\\ |R| = 2^{-n}}} \varepsilon_{R} h_{R} (0,x_2)  \bigg\|_\infty =   \bigg\|  \sum_{|R|  =  2^{-n}} \varepsilon_R h_R   \bigg\|_\infty = n +1 .
 \end{equation}
 We do not know whether estimates \eqref{e.l2a} and \eqref{e.l2} are generally equivalent, i.e. whether Lemma \ref{l.reduce} can be reversed in dimension $d\ge 3$.


Summation over the set $\{ |R|\ge 2^{-n} \}$ has, in fact, been considered in \cite{teml}, where it was shown that in dimension $d=2$ 
\begin{equation}\label{e.tem}
 \bigg\| \sum_{ |R| \ge 2^{-n} }  \alpha_R  h_R \bigg\|_\infty \gtrsim 2^{-n} \sum_{|R|=2^{-n} } |\alpha_R |,
\end{equation}
which is slightly stronger than the two-dimensional version \eqref{e.sbi2} of the  {\it{small ball conjecture}} \eqref{e.sbc}: the right-hand side stays the same, but the left-hand side contains more terms. Observe that the signed version of  estimate \eqref{e.tem}, in all dimensions $d'\ge 1$, is trivial due to  orthogonality:
\begin{align}
\nonumber \bigg\|  \sum_{|R| \ge 2^{-n}  } \varepsilon_R h_R  \bigg\|_\infty & \ge  \bigg\|  \sum_{|R| \ge 2^{-n}  } \varepsilon_R h_R  \bigg\|_2  
\nonumber  = \Bigg( \sum_{k=0}^n \sum_{|R|=2^{-k} } 2^{-k} \Bigg)^{\frac12} \approx \Bigg( \sum_{k=0}^n  k^{d'-1} \Bigg)^{\frac12} \approx n^{d'/2},
\end{align}
where we have used the fact that $\# \{ R: \, |R|=2^{-k} \} \approx k^{d'-1}\cdot 2^k$. This naturally only implies the  $L^2$ bound for the signed small ball inequality in dimension $d=d'+1$, i.e.  $\displaystyle{ \bigg\|  \sum_{|R| \ge 2^{-n}  } \varepsilon_R h_R  \bigg\|_\infty \gtrsim n^{\frac{d-1}{2}} } $. The signed version of \eqref{e.tem} implies that one neither gains, nor loses by extending the range of summation from $\{ |R| = 2^{-n} \}$ to $\{ |R| \ge 2^{-n} \}$. In order to deduce the three-dimensional version of the conjecture, one should show that this extension actually gains  a factor of $\sqrt{n}$. This may be a viable approach to the three-dimensional signed small ball inequality, since one can use two-dimensional methods, however many of the difficulties are still preserved.




\subsection{The case of general coefficients} One may ask whether this ``dimension reduction" applies to the general small ball inequality \eqref{e.sbc}. We shall demonstrate that, while such a  reduction is formally possible, and moreover, the underlying one-dimensional estimate even yields a proof of the signed conjecture in all dimensions, there are intrinsic problems: already the one-dimensional case is false!!! This reveals   some  fundamental differences between the general and the signed inequalities.

In dimension $d'=1$ a proper analog of \eqref{e.l2a}, or of \eqref{e.s1d}, would be the following: 
\begin{equation}\label{e.2b}
\bigg\| \sum_{\substack{ I \in \mathcal D:\,  |I | \ge 2^{-n}}  } \alpha_I h_I \bigg\|_\infty \gtrsim \sum_{  |I | \ge 2^{-n}  }  \big| \alpha_I \big| \cdot |I|.
\end{equation}
This inequality  would easily  imply  the small ball inequality  \eqref{e.sbc}  in two dimensions. Indeed, fix $x_2$ for the moment. Then
\begin{align*}
\bigg\|  \sum_{|R| = 2^{-n} } \alpha_R h_R \,\, \bigg\|_{L^\infty (x_1) } &  = \bigg\| \sum_{|R_1|\ge 2^{-n}} \bigg( \sum_{|R_2| = 2^{-n}/|R_1|}  \alpha_{R_1\times R_2} h_{R_2} (x_2) \bigg) h_{R_1} (x_1) \bigg\|_{L^\infty(x_1)}\\
& \ge \sum_{|R_1|\ge 2^{-n}} \,\,  \bigg| \sum_{|R_2| = 2^{-n}/|R_1|}  \alpha_{R_1\times R_2} h_{R_2} (x_2) \bigg| \cdot |R_1|.
\end{align*}
Replacing the sup by the average in the second variable we find that
\begin{align*}
\bigg\|  \sum_{|R| = 2^{-n} } \alpha_R h_R \,\, \bigg\|_\infty  & \ge \sum_{|R_1|\ge 2^{-n}} \,\,  \bigg\| \sum_{|R_2| = 2^{-n}/|R_1|}  \alpha_{R_1\times R_2} h_{R_2}  \bigg\|_{L^1 (x_2) } \cdot |R_1|  \\
&  = \sum_{|R_1|\ge 2^{-n}} \,\,   \sum_{|R_2| = 2^{-n}/|R_1|}  |\alpha_{R_1\times R_2} | \cdot |R_2|  \cdot |R_1|  = 2^{-n} \sum_{|R|= 2^{-n} } |\alpha_R |,
\end{align*}
where we have used the fact that all the Haar functions inside the $L^1$ norm have disjoint support. Higher-dimensional analogs of this implication may also be formulated. 

\subsection{The one-dimensional estimate \eqref{e.2b} implies the signed small ball conjecture in all dimensions} Quite unexpectedly, the one-dimensional inequality \eqref{e.2b} would  actually imply the signed small ball inequality \eqref{e.ssbc} in {\emph{all dimensions}} $d\ge 2$ via the following argument.


Denote  $\displaystyle{  H_n =  \sum_{|R| = 2^{-n}} \varepsilon_R h_R}$.  We notice that  in dimension $d\ge 2$,  $  \| H_n \|_1 \gtrsim n^{\frac{d-1}2 }$. Indeed, we have $\| H_n \|_2 \approx \|H_n\|_4 \approx n^{\frac{d-1}2}$ ($L^2$ is just Parseval's identity, and $L^4$ follows from Littlewood--Paley: this is well explained in \cite{bilyk}). And a simple application of H\"{o}lder's inequality, namely $\|H_n\|_2 \le \|H_n\|_1^{1/3} \cdot \| H_n \|_4^{2/3}$, yields the answer. We are going to use this observation in dimension $d-1$.\\

Write $x\in [0,1)^d$ as $x=(x_1, x')$, where $x_1\in [0,1)$, $x' \in [0,1)^{d-1}$. Also, split any $R = R_1 \times R' \in \mathcal D^d$ in a similar way. Fix $x'$ for the moment. We have, using \eqref{e.2b}, that

\begin{align*}
\bigg\|  \sum_{|R| = 2^{-n} } \varepsilon_R h_R \,\, \bigg\|_{L^\infty (x_1) } & = \bigg\|  \sum_{|R_1| \ge 2^{-n} } \bigg( \sum_{|R'| = 2^{-n}/|R_1|}  \varepsilon_{R_1 \times R'} h_{ R'} (x') \bigg) h_{R_1} (x_1) \,\, \bigg\|_{L^\infty (x_1)}\\
& \ge \sum_{|R_1| \ge 2^{-n} } \bigg| \sum_{|R'| = 2^{-n}/|R_1|}  \varepsilon_{R_1 \times R'} h_{ R'} (x') \bigg| \cdot |R_1| \\
\end{align*}
We now replace the $L^\infty$ norm in the rest of the variables by $L^1$ and then use the observation above.

\begin{align*}
\bigg\|  \sum_{|R| = 2^{-n} } \varepsilon_R h_R \,\, \bigg\|_\infty & \ge  \sum_{|R_1| \ge 2^{-n} } \bigg\| \sum_{|R'| = 2^{-n}/|R_1|}  \varepsilon_{R_1 \times R'} h_{ R'} (x') \bigg\|_{L^1(x')} \cdot |R_1|\\
& \gtrsim \sum_{k=0}^n \sum_{|R_1| = 2^{-k} }  (n-k)^{\frac{d-2}{2}} \cdot 2^{-k} = \sum_{k=0}^n (n-k)^{\frac{d-2}{2}} \approx n^{d/2}.
\end{align*}

\noindent {\emph{Remark.}} The numerology of the signed small ball conjecture says that, roughly speaking, one should get $\sqrt{n}$ from every dimension. An $L^2$ estimate only gives one square root less. What we do here is apply the one-dimensional estimate \eqref{e.2b} to obtain $n$ from the first coordinate, then apply the $L^2$ bounds to get exponent $\frac{d-2}{2}$ from the other $d-1$ coordinates. A similar argument has recently been used in \cite{karslidis} to prove the signed conjecture in all dimensions, in the case when the signs have product structure $\varepsilon_R = \varepsilon_{R_1} \cdot \varepsilon_{R'}$. Moreover, a very similar idea has been exploited in \cite{skrig} to show that the discrepancy conjecture \eqref{e.dc2} holds for a random digit-shift of an arbitrary $N$-point set in $[0,1)^d$.

\subsection{The one-dimensional estimate \eqref{e.2b} is false.}\label{s.1dfalse} We shall now demonstrate that unfortunately the one-dimensional inequality \eqref{e.2b} does not hold in general. This sharply contrasts the signed case, where the corresponding inequality \eqref{e.s1d} is almost trivial.

The counterexample is constructed inductively with all coefficients $\alpha_I$ equal to zero or one. Start with the scale $k=0$ and assign $\alpha_{[0,1)}=1$. At each subsequent scale $k$, for every dyadic interval $J$ with $|J| = 2^{-k}$ and with not yet assigned coefficient, consider the modulus of the value of the sum of the previously chosen signed Haar functions $\displaystyle{\bigg| \sum_{|I|>2^{-k}} \alpha_I h_I \bigg|}$ on the interval $J$. If this value is at least $n^\frac23$, assign coefficient zero to $J$ as well as all of its ``children", i.e. $\alpha_J = 0$ and $\alpha_{J'} = 0$ for all $J'\subset J$. Otherwise, assign $\alpha_J = 1$.

In the end, by construction we would have that the left-hand side satisfies
\begin{equation}
\bigg\| \sum_{\substack{   |I | \ge 2^{-n}}  } \alpha_I h_I \bigg\|_\infty \le n^\frac23,
\end{equation}
since we suppressed all intervals, where the value may have been greater.

On the other hand, for the right-hand side we have 
\begin{equation}
\sum_{  |I | \ge 2^{-n}  }  \big| \alpha_I \big| \cdot |I| = \sum_{  |I | \ge 2^{-n}  }  \big| \alpha_I \big|^2 \cdot |I| = \bigg\| \sum_{\substack{   |I | \ge 2^{-n}}  } \alpha_I h_I \bigg\|_2^2
\end{equation}
by Parseval's identity. 


 If all coefficients are equal to one, the sum $\displaystyle{\sum_{\substack{   |I | \ge 2^{-n}}  } h_I}$ may be viewed as   a sum of $n+1$ independent $\pm 1$ random variables $\displaystyle{\sum_{\substack{   |I | = 2^{-k}}  } h_I}$,  where $k=0,1,\dots,n$. Thus, there exists a set of positive measure $c >0$, on which $\displaystyle{\bigg| \sum_{\substack{   |I | \ge 2^{-n}}  } h_I \bigg|  > c_1 \sqrt{n}}$.  
On those intervals, where we have set the coefficients equal to zero, by construction, we have $\displaystyle{\bigg| \sum_{\substack{   |I | \ge 2^{-n}}  } \alpha_I  h_I \bigg|  \ge n^{\frac23}\gg \sqrt{n}}$. Therefore, we still have $\displaystyle{\bigg| \sum_{\substack{   |I | \ge 2^{-n}}  }  \alpha_I h_I \bigg|  > c_1 \sqrt{n}}$ on a set of positive measure, i.e. $\displaystyle{\bigg\| \sum_{\substack{   |I | \ge 2^{-n}}  } \alpha_I h_I \bigg\|_2 \gtrsim \sqrt{n}}$.

Thus the left-hand side is at most $n^\frac23$, while the right-hand side is of the order $n \gg n^\frac23$. Therefore \eqref{e.2b} fails in this case.\\

{\bf{Acknowledgements:}} The authors would like to express gratitude to IMA and NSF: the research of the first author was supported by the NSF grant DMS 1101519, and the second author's stay at IMA was also supported by NSF funds.  The authors are also extremely  grateful to Ohad Feldheim  for contributing the idea of \S \ref{s.1dfalse}, and to Josef Dick and  Wolfgang Schmid for pointing out references \cite{xiao1,xiao2}.

\end{document}